\documentclass[oneside,12pt]{amsart}

\usepackage{amsmath,amssymb,amsthm,amsfonts,amstext,amsbsy,amscd,mathrsfs,graphics,graphicx,geometry,color}
\usepackage[colorlinks=true, pdfstartview=FitV, linkcolor=blue, citecolor=blue, urlcolor=blue,pagebackref=false]{hyperref}
\usepackage{amsxtra}
\usepackage[english]{babel}

\usepackage[T1]{fontenc}
\usepackage{latexsym}
\usepackage[latin1]{inputenc}
\newtheorem{theorem}{Theorem}[section]
\newtheorem{lemma}[theorem]{Lemma}

\newtheorem{corollary}[theorem]{Corollary}
\newtheorem{proposition}[theorem]{Proposition}
\newtheorem{example}[theorem]{Example}
\newtheorem{remark}[theorem]{Remark}

\newtheorem{definition}[theorem]{Definition}
\def\bit{\begin{itemize}}
\def\eit{\end{itemize}}
\reversemarginpar   
\def\bc{\begin{center}}
\def\ec{\end{center}}
\def\bthm{\begin{theorem}}
\def\ethm{\end{theorem}}
\def\bcor{\begin{corollary}}
\def\ecor{\end{corollary}}
\def\bprop{\begin{proposition}}
\def\eprop{\end{proposition}}
\def\blem{\begin{lemma}}
\def\elem{\end{lemma}}

\def\brem{\begin{remark}}
\def\erem{\end{remark}}

\def\bdes{\begin{description}}
\def\edes{\end{description}}

\def\beq{\begin{equation}}
\def\eeq{\end{equation}}

\def\ben{\begin{enumerate}}
\def\een{\end{enumerate}}

\def\beqar{\begin{eqnarray}}
\def\eeqar{\end{eqnarray}}
\def\beqarr{\begin{eqnarray*}}
\def\eeqarr{\end{eqnarray*}}

\def\RR{{\mathbb R}}  

\def\EE{{\mathbb E}}
\def\PP{{\mathbb P}}

\def\cA{\mathcal{A}} \def\cB{\mathcal{B}} 
\def\cD{\mathcal{D}}  \def\cF{\mathcal{F}}

\def\cP{\mathcal{P}}  
\def\cS{\mathcal{S}}

\def\part{\partial}

\def\d#1dt{\frac{d#1}{dt}}    

\begin{document}
\title{Application of stochastic flows to the sticky Brownian motion equation}
\maketitle
\begin{center}
\renewcommand{\thefootnote}{(\arabic{footnote})}
  \scshape Hatem Hajri \footnote{Laboratoire de l'Int\'egration du Mat\'eriau au Syst\`eme, Bordeaux. Email:
Hatem.Hajri@ims-bordeaux.fr}
   Mine Caglar \footnote{ Koc University, Istanbul. Email:
mcaglar@ku.edu.tr \\ The research of M. Caglar is supported by TUBITAK Project No. 115F086}
 Marc Arnaudon \footnote{ Institut de Math\'ematiques de Bordeaux, Bordeaux. Email:
marc.arnaudon@math.u-bordeaux.fr}\setcounter{footnote}{0}
\end{center}

\begin{abstract}
We show how the theory of stochastic flows allows to recover in an elementary way a well known result of Warren on the sticky Brownian motion equation.
\end{abstract}







\section{Introduction}
A $\theta$-sticky Brownian on the half line $[0,\infty)$ is a diffusion with generator
$$(\mathcal A f)(x)=\begin{cases} \frac{1}{2} f''(x) &\mbox{if } x>0 \\
 \theta f'(0+) & \mbox{if } x=0\end{cases} $$
and domain
$$\cD(\cA) = \bigg\{f\in C^2(0,\infty):f\in C_0([0, \infty)), f'(0^+), $$
$$\ \ \ \ \ \ \ \ \ \ \ \ \ \ \ \ \ \ \ \ \ \ \ \ \ \ \ \ \ \ \  f''(0^+) \mbox{ exist}, f''(0^+) =2\theta f'(0^+), \lim_{x\rightarrow \infty} f''(x)=0 \bigg\}$$
where $\theta>0$ is the stickiness parameter. This is a special case of Feller one dimensional diffusions introduced by Feller by means of their infinitesimal generators \cite{MR0047886}. For comparison of the boundary condition $f''(0^+) =2\theta f'(0^+)$   with other examples, see \cite{MR3271518}. Sticky Brownian motion has an intermediate behavior, depending on $\theta$, between Brownian motion absorbed at $0$ and reflected Brownian motion. One possible path construction of a $\theta$-sticky Brownian motion $X$ started from $0$ consists in slowing down a reflected Brownian motion $R$ started from $0$ whenever it is at $0$ in the following way
$$X_t=R_{\text{inf}\{u : u + \frac{1}{\theta} L_u>t\}}$$
where $L_t=\lim_{\epsilon\rightarrow 0} \frac{1}{2\epsilon} \int_{0}^{t} 1_{\{0\le R_s\le \epsilon\}} ds$ is the local time of $R$ \cite{howit,MR0345224}. As a consequence of this construction, the amount of time spent at $0$ by $X$ up to $t$, $\int_{0}^{t} 1_{\{X_s = 0\}} ds$, has positive probability of being greater than $0$. More precisely, the following equality holds in law
$$\int_{0}^{t} 1_{\{X_s = 0\}} ds \overset{\text{law}}{=} \frac{|N|}{\theta} \sqrt{t+ \frac{N^2}{4\theta^4}} - \frac{N^2}{2\theta^2}$$
where $N\sim \mathcal N(0,1)$ (see Proposition 5 in \cite{howit}).

In this paper, we are interested in sticky Brownian motion as solution of the following stochastic differential equation
\begin{equation}\label{eq}
X_t = x + \int_{0}^{t} 1_{\{X_s>0\}} dW_s + \theta \int_{0}^{t} 1_{\{X_s = 0\}} ds
\end{equation}
driven by a standard Brownian motion $W$ and where $\theta>0$ is a given constant and $x\in[0,\infty)$ is a given initial condition.

It has been proved by Chitashvili \cite{MR1639096} that (\ref{eq}) has a weak solution $X$ which is a $\theta$-sticky Brownian motion started from $x$, the law of $(X,W)$ is unique but $X$ is not a strong solution to (\ref{eq}). Later on, Warren \cite{Warren1997} derived the following remarkable result describing the law of $X_t$ given $W$ (the form given here follows Theorem 2 \cite{MR1767999}).

\begin{theorem}\label{ty} When $X_0=0$, for all $t\ge 0$ and all measurable bounded $f$,
\[
\EE [f(X_t)|\cF^W]=G_f(W_t^+)
\]
where $W_t^+ = W_t - \min_{0\le u \leq t} W_u$ and $G_f(y)=\EE[f((y-T)^+)]$ with $T$ an exponential variable with mean $\frac{1}{2\theta}$.
\end{theorem}
This theorem shows, in particular, that $X$ cannot be a strong solution to (\ref{eq}). Subsequently, Warren \cite{MR1767999} described all couplings of solutions to (\ref{eq}) which leave the diagonal. Before going on, we mention the work of Engelbert and Peskir \cite{MR3271518} where a third proof of the non strong solvability of (\ref{eq}) and a two sided version of it can be found (see also \cite{EJP2350}).

A remarkable and attractive fact in Warren's conditional law identity is that it involves the well known and habitual process $W^+$ strong solution to
\begin{equation}\label{refl}
Y_t = W_t + L_t(Y)
\end{equation}
where $L_t(Y)=\lim_{\epsilon\rightarrow 0} \frac{1}{2\epsilon} \int_{0}^{t} 1_{\{0\le Y_s\le \epsilon\}} ds$. This raises the question whether there is a link between (\ref{eq}) and (\ref{refl}) explaining Theorem \ref{ty}.

In this paper, it is shown that stochastic flows of kernels \cite{MR2060298} provide an answer to the previous question. More precisely, define
\begin{equation}\label{ref}
\varphi_{s,t}(x) = (x + W_t - W_s) 1_{\{t\le \tau_s(x)\}} + W^+_{s,t} 1_{\{t>\tau_s(x)\}}
\end{equation}
where $\tau_s(x)=\inf\{u\ge s : x + W_u - W_s = 0\}$. Then $\varphi$ is a stochastic flow of maps which solves the flow version of (\ref{refl})
$$\varphi_{s,t}(x) = x + W_t - W_s + L_{s,t}(x)$$
where $L_{s,t}(x) = \lim_{\epsilon\rightarrow 0} \frac{1}{2\epsilon} \int_{s}^{t} 1_{\{0\le \varphi_{s,u}(x)\le \epsilon\}} du$. Now, let
\begin{equation}\label{ftt}
K_{s,t} f(x)= \begin{cases} f(\varphi_{s,t}(x)) &\mbox{if } s\le t\le \tau_s(x) \\
 G_f(\varphi_{s,t}(x)) & \mbox{if } t>\tau_s(x)\end{cases}
\end{equation}
Then $K$ is a stochastic flow of kernels which is a strong solution to the flow of kernels version of (\ref{eq}): for all $t\ge s, f\in\mathcal D(A)$ and $x\ge 0$ a.s.
\begin{equation}\label{trt}
K_{s,t} f(x) = f(x) + \int_{s}^{t} K_{s,u}(f'1_{(0,\infty)})(x) dW_u+ \frac{1}{2} \int_{s}^{t} K_{s,u} f''(x) du
\end{equation}
$K$, called the Wiener solution of (\ref{trt}) in \cite{MR2060298}, is characterized by being the unique (up to modification), strong solution of (\ref{trt}). This leads to Theorem (\ref{ty}) as the conditional law of $X_t$ given $W$ should coincide with $K_{0,t}(0,dy)$. Note that Equation (\ref{trt}) encapsulates the flow property (iv) of Definition \ref{defsfk} for $K$. Therefore, identifying the complete flow $K_{s,t}(x,dy)$ for every $s,t$ and $x$ is crucial in proving this result, not only for $s=0$ and $x=0$. The semigroup and Feller properties also play an important role in this fact. See \cite{MR2060298} for further discussion on $ \EE [f(X_t)|\cF^W]$ satisfying Equation (\ref{trt}). As a complete proof, we argue that (\ref{ftt}) being the Wiener flow satisfies this equation and the theorem follows for $G_f(W_t^+)$ in the special case $X_0=0$.
Section 2 gives details and proofs of the previously claimed facts. It can be remarked that the proofs only rely on the definition of stochastic flows with no additional results of the theory.
 The present paper provides, in particular, a direct application of stochastic flows to the study of weak solutions (see \cite{MR292205755} for another recent application). In Section 3, we conclude the paper with the Wiener chaos expansion of the conditional law.

\section{Proof of Theorem \ref{ty}} \label{2}
\subsection{The generalized sticky Brownian motion equation}
\leavevmode\par
\vspace{0.2cm}
Let us now recall the definition of stochastic flows from \cite{MR2060298}. In this definition $\cP(\RR_+)$ denotes the space of all probability measures on $\RR_+$ and $\cB(E)$ indicates the Borel $\sigma$-field of $E$.
\begin{definition} \label{defsfk}
A stochastic flow of kernels $K$ on $\RR_+$, defined on a probability space $(\Omega,\cA,\PP)$, is a family $(K_{s,t})_{s\le t}$ such that
\begin{enumerate}
\item For all $s\le t$, $K_{s,t}$ is a measurable mapping from $(\RR_+\times\Omega,\cB(\RR_+)\otimes\cA)$ to $(\cP(\RR_+),\cB(\cP(\RR_+)))$;
\item For all $h\in\RR$, $s\le t$, $K_{s+h,t+h}$ is distributed like $K_{s,t}$;
\item For all $s_1\le t_1\le\cdots \le s_n\le t_n$, the family $\{K_{s_i,t_i}, 1\le i\le n\}$ is independent;
\item For all $s\le t\le u$ and all $x\in \RR_+$, a.s. $K_{s,u}(x)=K_{s,t}K_{t,u}(x)$ and $K_{s,s}(x)=\delta_x$;
\item For all $f\in C_0(\RR_+)$, $s\le t$,
$$\lim_{(u,v)\to (s,t)}\sup_{x\in \RR_+}\EE[(K_{u,v}f(x)-K_{s,t}f(x))^2]=0;$$
\item For all $f\in C_0(\RR_+)$, $x\in \RR_+$, $s\le t$,
$$\lim_{y\to x}\EE[(K_{s,t}f(y)-K_{s,t}f(x))^2]=0;$$
\item For all $s\le t$, $f\in C_0(\RR_+)$, $\lim_{x\to\infty}\EE[(K_{s,t}f(x))^2]=0$.
\end{enumerate}
\end{definition}
We say that $\varphi$ is a stochastic flow of mappings on $\RR_+$ if $K_{s,t}(x)=\delta_{\varphi_{s,t}(x)}$ is a stochastic flow of kernels on $\RR_+$.

For $K$, a stochastic flow of kernels on $\RR_+$,
\begin{equation}\label{tg}
P^n_tf(x_1,\cdots,x_n)=\mathbb E\left[\int_{\RR_+^n} f(y_1,\cdots,y_n)K_{0,t}(x_1,dy_1)\cdots K_{0,t}(x_n,dy_n)\right]
\end{equation}
defines a Feller semigroup on $\RR_+^n$. Moreover $(P^n)_{n\ge 1}$ is a compatible family (in a sense explained in \cite{MR2060298}) of Feller semigroups acting respectively on $C_0(\RR_+^n)$ that uniquely characterize  the law of $K$. Conversely, it has been proved in \cite{MR2060298} that to each family of compatible Feller semigroups $(P^n)_{n\ge 1}$ is associated a (unique in law)  stochastic flow of kernels such that (\ref{tg}) holds for every $n\ge 1$.
\begin{definition}(Real white noise)
A family $(W_{s,t})_{s\le t}$ is called a real white noise if there exists a Brownian motion on the real line  $(W_t)_{t\in\RR}$, that is $(W_t)_{t\geq 0}$ and $(W_{-t})_{t\geq 0}$ are two independent standard Brownian motions such that for all $s\le t$, $W_{s,t}=W_t-W_s$ (in particular, when $t\ge 0$, $W_t=W_{0,t}$ and $W_{-t}=-W_{-t,0}$).
\end{definition}
For a family of random variables $Z=(Z_{s,t})_{s\le t}$, define $\mathcal F^Z_{s,t}=\sigma(Z_{u,v}, s\leq u\leq v\leq t)$ for all $s\le t$.


\begin{definition}\label{def}
Let $K$ be a stochastic flow of kernels and $W$ be a real white noise defined on the same probability space. We say that $(K,W)$ is a (generalized) solution of the sticky equation if for all $f\in \cD(\cA)$, $t\ge s$ and $x\in{\mathbb R}_+$ a.s.
$$K_{s,t}f(x)=f(x) + \int_{s}^{t} K_{s,u} (f'1_{(0,\infty)})(x) dW_u + \frac{1}{2} \int_{s}^{t} K_{s,u} f''(x) du $$
\end{definition}
Let us explain the link between this equation and the original sticky equation (\ref{eq}). We start with the following
\begin{lemma} \label{filt}
If $(K,W)$ is a solution of the generalized sticky equation, then
\[
 K_{s,t} f(x)=   f(\varphi_{s,t}(x)) \quad \mbox{if } \quad s\le t\le \tau_s(x)
\]
and in particular, $\mathcal F^W_{s,t}\subset\mathcal F^K_{s,t}$ for all $s\le t$.
\end{lemma}

\begin{proof} Let $W_{s,t}:=W_t-W_s$ for $0\le s\le t$. Fix $x>0$ and, for small $\epsilon\in]0,x[$, define $ \tau_s^{\epsilon}(x)=\inf\{u\ge s : x + W_{s,u} = \epsilon\}$. Then along the same lines of the proof of Lemma 3.1 in \cite{MR2235172}, one can show that   $K_{s,t}(x)=\delta_{x+W_{s,t}}$ for all $s\le t\le \tau_s^{\epsilon}(x)$. As $\epsilon>0$ is arbitrarily small, $K_{s,t}(x)=\delta_{x+W_{s,t}}$ also for $t\le \tau_s(x):=\inf\{u\ge s : x + W_{s,u} = 0\}$. Since this holds for $x$ arbitrarily distant from $0$, the lemma follows.
\end{proof}

In view of Lemma \ref{filt}, we may sometimes say $K$ is a solution of the generalized sticky equation without specifying the white noise since it is determined by $K$.

Assume now $(K,W)$ satisfies Definition \ref{def} and set
$$Q_t(f\otimes g)(x,w)=\EE[K_{0,t}f(x) g(w+W_t)]$$
By the previous lemma, $(Q_t)_t$ defines a Feller semigroup. Denote by $\mathcal L$ its generator and $\mathcal D(\mathcal L)$ its domain. A simple application of It\^o's formula shows that
$\mathcal D_1\otimes C^2_K(\mathbb R)\subset \mathcal D(\mathcal L)$ where $\mathcal D_1=\{f\in\mathcal D(\mathcal A) : f'(0+)=0\}$ and $C^2_K(\mathbb R)$ denotes the space of $C^2$ functions on $\RR$ with compact supports. Moreover for all $f\in \mathcal D_1$ and $g\in C^2_K(\mathbb R)$,
$$\mathcal L(f\otimes g)(x,w)=\frac{1}{2} f(x) g''(w) + \frac{1}{2} g(w) f''(x) + f'(x) g'(w).$$
Let $(X,B)$ be the Markov process associated to $(Q_t)_t$ and started from $(x,0)$. Then $X$ is a $\theta$-sticky Brownian motion started from $x$ and $B$ is a standard Brownian motion started from $0$.
Now for $f\in \mathcal D_1$ and $g\in C^2_K(\mathbb R)$,
\begin{equation}\label{martinga}
f(X_t) g(B_t) - \int_{0}^{t} \mathcal L(f\otimes g)(X_s,B_s)\ \text{is a martingale.}
\end{equation}
 As $X$ is a $\theta$-sticky Brownian motion, it satisfies $X_t=x + M_t + \theta \int_{0}^{t} 1_{\{X_s=0\}} ds$ with $M$ a martingale with quadratic variation $\langle M\rangle_t = \int_{0}^{t} 1_{\{X_s>0\}} ds$. Writing It\^o's formulas for $f(X) g(B)$ and using (\ref{martinga}) shows that
\begin{equation}  \label{apparg}
\int_{0}^{t} f'(X_s) g'(B_s) d\langle M, B\rangle_s = \int_{0}^{t} f'(X_s) g'(B_s) ds.
\end{equation}
Now, one can find a sequence $(f_n)\subset \mathcal D_1$  such that $f_n'(x)\to 1_{(0,\infty)}(x)$  as $n\to \infty$ for each $x>0$ and $\sup_{x}f'_n(x)\le 1$. On the other hand, there exists a sequence $(g_n) \subset C^2_K(\mathbb R)$ such that the support of $g_n$ is $[-n,n]$  with $\sup_{x} g_n'(x)\le 1$ and $g_n'(x) \to 1 $ as $n\to \infty$ for each $x\in \mathbb R$. In view of (\ref{apparg}), it follows from bounded convergence theorem that
$1_{\{X_s>0\}} d\langle M, B\rangle_s = 1_{\{X_s>0\}} ds$ since the integrals $\int_{0}^{t} f_n'(X_s) g_n'(B_s) ds$ are bounded by $t$, for each $t>0$. Consequently, $M_t = \int_{0}^{t} 1_{\{X_s>0\}} dB_s$ is in $L^2(\mathbb P)$. So finally, $(X,B)$ is a weak solution to the sticky equation.

More generally, considering the semigroups $Q^n_t(f\otimes g)(x,w)=\EE[K^{\otimes n}_{0,t}f(x) g(w+W_t)]$ for $n\ge 1$, one can prove that there exists a one to one correspondence between the laws of stochastic flows of kernels satisfying Definition \ref{def} and compatible weak solutions to the sticky equation (see Proposition 2.1 in \cite{MROP} for more details in a similar context).

To close this subsection, we mention that if $\varphi$ is a flow of mappings such that $K=\delta_{\varphi}$ satisfies Definition \ref{def}, then necessarily $\varphi$ is a flow of mappings solution of
$$\varphi_{s,t}(x) = x + \int_{s}^{t} 1_{\{\varphi_{s,u}(x)>0\}} dW_u + \theta \int_{s}^{t} 1_{\{\varphi_{s,u}(x) = 0\}} ds$$
and vice versa. Warren \cite{EJP120} proved that such a flow $\varphi$ exists and its law is uniquely determined. This flow can also be constructed by applying the general results of \cite{MR2060298}.

\subsection{The Wiener flow}
\begin{definition}
Let $(K,W)$ be a solution of the generalized sticky equation. If $\mathcal F^K_{s,t}\subset \mathcal F^W_{s,t}$ for all $s\le t$, then $(K,W)$ is called a Wiener solution.
\end{definition}
The Wiener solution was introduced in \cite{MR1905858} (it is called statistical solution there) by means of its Wiener chaos expansion with respect to $W$ only depending on the semigroup of the diffusion (sticky Brownian motion here). Interestingly, this solution exists and is unique under weak assumptions.
\begin{proposition}
Let $(K^1,W)$ and $(K^2,W)$ be two Wiener solutions of the sticky equation relatively to the same Brownian motion. Then for all $s\le t$ and $x\in\mathbb R$, with probability one, $K^1_{s,t}(x)=K^2_{s,t}(x)$.
\end{proposition}
\begin{proof}
We follow the proof of Proposition 4.2 \cite{MROP}. Note that $K_{s,t}(x,y)=K^1_{s,t}(x)\otimes K^2_{s,t}(y)$ is a stochastic flow of kernels on $\RR_+^2$ and
$$Q_t(f\otimes g\otimes h)(x,y,w):=\EE[K^1_{0,t}f(x)K^2_{0,t}g(y)h(w+ W_t)]$$
is a Feller semigroup on $(\RR_+)^2\times \RR$. Fix $x\in\RR_+$ and let $(X^1,X^2,B)$ be the Markov process associated to $Q$ started from $(x,x,0)$, then $B$ is a standard Brownian motion and $X^1, X^2$ are two $\theta$-sticky Brownian motions. Moreover $X^1, X^2$ are solutions of the sticky equation driven by $B$ and in particular $(X^1,B)$ and $(X^2,B)$ have the same law. Since $K^1$ and $K^2$ are two Wiener solutions, there exist two measurable functions $F^1_{t,x}, F^2_{t,x}:C([0,t],\RR)\rightarrow\mathcal P(\RR)$ such that $K^1_{0,t}(x)=F^1_{t,x}(W_u, u\le t), K^2_{0,t}(x)=F^2_{t,x}(W_u, u\le t)$. Let $N^1_{0,t}(x) = F^1_{t,x}(B_u, u\le t)$ and $N^2_{0,t}(x)=F^2_{t,x}(B_u, u\le t)$. We will prove that for all measurable bounded $f:\RR\rightarrow\RR$ a.s.
\begin{equation}\label{yes}
N^i_{0,t}f(x) = \EE[f(X^i_t)|\sigma(B_u, u\le t)],\ \ i=1,2
\end{equation}
To prove (\ref{yes}), we will check  by induction on $n$ that for all $t_1\le\cdots\le t_{n-1}\le t_n=t$ and all bounded functions $f, g_1,\cdots,g_n:\RR\rightarrow\RR$, we have
\begin{equation}\label{rassoul}
\EE\big[K^i_{0,t} f(x)\prod_{j=1}^{n} g_j(W_{t_j})\big]=\EE\big[f(X^i_t)\prod_{j=1}^{n} g_j(B_{t_j})\big], \ \ \ i=1,2.
\end{equation}
Let us prove this for $i=1$ and set $Q^1_t(f\otimes g) = Q_t(f\otimes \text{Id}\otimes g)$. For $n=1$, (\ref{rassoul}) is immediate from the definition of $Q$. Let us prove (\ref{rassoul}) for $n=2$. We have
$$\EE[K^1_{0,t} f(x)g_1(W_{t_1})g_2(W_{t})] = \EE[K^1_{0,t_1}(Q^1_{t - t_1}(f\otimes g_2)(\cdot,W_{t_1}))(x) g_1(W_{t_1})].$$
On the other hand
$$\EE[f(X^1_t)g_1(B_{t_1})g_2(B_{t})] = \EE[Q^1_{t - t_1}(f\otimes g_2)(X^1_{t_1},B_{t_1})g_1(B_{t_1})].$$
Now (\ref{rassoul}) holds using a uniform approximation of $Q^1_{t_2 - t_1}(f\otimes g)$ by a linear combination of functions of the form $h\otimes k$, $h, k\in C_0(\RR_+)$. It is clear now from (\ref{yes}), that $N^1_{0,t}(x) = N^2_{0,t}(x)$ since $(X^1,B)$ and $(X^2,B)$ have the same law.

\end{proof}
Now in the rest of the paper, we take $W$ a real white noise and will check that $K$ defined in (\ref{ftt}) is the Wiener solution of the generalized sticky equation. This gives Theorem (\ref{ty}) in view of what precedes.

\begin{proposition}
$K$ is the, unique up to modification, Wiener stochastic flow of kernels solution of the generalized sticky equation driven by $W$.
\end{proposition}
\begin{proof}
To check that $K$ is a stochastic flow of kernels, we will only check the flow property for all $f\in C_0(\mathbb R_+)$, $s\le t\le u$ and $x\in\RR_+$, with probability $1$,
\begin{equation}\label{flow}
K_{s,u}f(x)=K_{s,t}K_{t,u}f(x)
\end{equation}

The other claims in Definition \ref{defsfk} are easy to verify. Let us now check (\ref{flow}). For this, we will use the fact that $\varphi$ defined in (\ref{ref}) is a stochastic flow of mappings (for non specialists of stochastic flows, this is a rather simple exercise). Note that $G_f$ writes as (with $\lambda=2\theta$)
\begin{equation}\label{h}
G_f(y)=f(0) e^{-\lambda y} + \lambda e^{-\lambda y} \int_{0}^{y} f(u) e^{\lambda u} du
\end{equation}
We first check (\ref{flow}) for $x=0$. By the flow property of $\varphi$, $K_{s,u}f(0)=G_f(\varphi_{s,u}(0))=G_f(\varphi_{t,u}\circ \varphi_{s,t}(0))$ and $K_{s,t}K_{t,u}f(0)=G_{K_{t,u}f}(\varphi_{s,t}(0))$. Using the independence of increments of $\varphi$, it suffices to prove that for all $y\ge 0$, a.s $G_f(\varphi_{t,u}(y))=G_{K_{t,u}f}(y)$ which is equivalent to
$$e^{\lambda y} G_f(\varphi_{t,u}(y))=G_f(\varphi_{t,u}(0))  + \lambda \int_{0}^{y} K_{t,u}f(a) e^{\lambda a} da$$
To prove this identity, note that for all $y>0$, $z\mapsto \varphi_{t,u}(z)$ is differentiable at $y$ with derivative given by $1_{\{u<\tau_t(y)\}}$. Thus by a simple calculation, the derivative of $z\mapsto e^{\lambda z}\varphi_{t,u}(z)$ at $y$ coincides with $\lambda e^{\lambda y} K_{t,u}f(y)$. This proves (\ref{flow}) for $x=0$.

Now take $x>0$ and let $y=\varphi_{s,t}(x)$.

On the event $\{u\le \tau_s(x)\}$, we have $\tau_s(x)=\tau_t(y)$, $K_{s,u}f(x)=f(\varphi_{s,u}(x))$, $K_{s,t} (K_{t,u}f)(x)=(K_{t,u}f) (y)= f(\varphi_{t,u}(y))$ since $u\le \tau_t(y)$ and so (\ref{flow}) holds by the flow property of $\varphi$.

On the event $\{t\le \tau_s(x)<u\}$, we still have $\tau_s(x)=\tau_t(y)$ and $K_{s,u}f(x)=G_f(\varphi_{s,u}(x))=G_f(\varphi_{t,u}(y))$. Moreover $K_{s,t} (K_{t,u}f)(x)=G_{K_{t,u}f}(y)$ and so the flow property holds by the calculations above.

On the event $\{\tau_s(x)\le t\}$, we have $K_{s,u}f(x)=G_f(\varphi_{s,u}(x))=G_f(\varphi_{s,u}(0))=K_{s,u}f(0)$. Moreover $K_{s,t} (K_{t,u}f)(x)=G_{K_{t,u}f}(y)=G_{K_{t,u}f}(\varphi_{s,t}(0))=K_{s,t} (K_{t,u}f)(0)$ and the flow property holds again from the case $x=0$.

Thus $K$ is a stochastic flow of kernels. Note that $\mathcal F^K_{s,t}\subset \mathcal F^W_{s,t}$ for all $s\le t$. It remains now to check that $K$ solves the generalized equation.
We take $s=0$ and first $x=0$. Denote $W^+_{0,t}$ simply by $W^+_t$. Let $$\mathcal D=\{g\in C^2(0,\infty):g\in C_0([0, \infty)), g'(0^+)=0, g''(0^+) \mbox{ exists}\}$$
By It\^o's formula, for all $g\in \mathcal D$
$$g(W^+_t)=g(0) + \int_{0}^{t} g'(W^+_u)dW_u + \frac{1}{2} \int_{0}^{t} g''(W^+_u) du $$
Let $f\in \cD(\cA)$ and set $g(y)=G_f(y)$. Then $g$ is continuous on $\mathbb R_+$, $C^2$ on $\mathbb R^{\ast}_+$ and
\begin{equation}\label{x}
g'(y)=-\lambda f(0) e^{-\lambda y} - \lambda^2 e^{-\lambda y} \int_{0}^{y} f(u) e^{\lambda u} du + \lambda f(y)
\end{equation}
In particular $g'(0+)=0$. Moreover
$$g''(y)=\lambda^2 f(0) e^{-\lambda y} + \lambda^3 e^{-\lambda y} \int_{0}^{y} f(u) e^{\lambda u} du - \lambda^2 f(y) + \lambda f'(y)$$
and so $G_f\in\mathcal D$. Consequently for all $f\in \cD(\cA)$,
\begin{eqnarray}
G_f(W^+_t)&=&G_f(0) + \int_{0}^{t} (G_f)'(W^+_u)dW_u + \frac{1}{2} \int_{0}^{t} (G_f)''(W^+_u) du \nonumber\\
&=&f(0) + \int_{0}^{t} (G_f)'(W^+_u)dW_u + \frac{1}{2} \int_{0}^{t} (G_f)''(W^+_u) du \nonumber\
\end{eqnarray}
We now check that for all $f\in \cD(\cA)$ and $y\ge 0$,
$$G_{f'1_{(0,\infty)}}(y) = (G_f)'(y)\ \text{and}\ \ G_{f''}(y) = (G_f)''(y)$$
Using (\ref{h}), we see that
$$G_{f'1_{(0,\infty)}}(y)=\lambda e^{-\lambda y} \int_{0}^{y} f'(u) e^{\lambda u} du$$
which is also equal to $(G_f)'(y)$ given in (\ref{x}) by a simple integration by parts. Again from (\ref{h}), we have
$$G_{f''}(y)=f''(0) e^{-\lambda y} + \lambda e^{-\lambda y} \int_{0}^{y} f''(u) e^{\lambda u} du$$
Integrating twice by parts, we see that
$$G_{f''}(y)=f''(0) e^{-\lambda y} - \lambda f'(0) e^{-\lambda y }+ \lambda f'(y) - \lambda^2 f(y) + \lambda^2 f(0) e^{-\lambda y} + \lambda^3 e^{-\lambda y} \int_{0}^{y} f(u) e^{\lambda u} du $$
which is the same as $(G_f)''(y)$ using the hypothesis $f''(0) = \lambda f'(0)$ as $f\in \cD(\cA)$.
Finally for all $f\in \cD_\cA$,
$$G_f(W^+_t)=f(0) + \int_{0}^{t} G_{f'1_{(0,\infty)}}(W^+_u)dW_u + \frac{1}{2} \int_{0}^{t} G_{f''}(W^+_u) du $$
or equivalently
$$K_{0,t}f(0)=f(0) + \int_{0}^{t} K_{0,u} (f'1_{(0,\infty)})(0) dW_u + \frac{1}{2} \int_{0}^{t} K_{0,u} f''(0) du$$
Now the case $x>0$ holds by discussing $t\le \tau_0(x)$ and $t>\tau_0(x)$ and using the fact that $K_{0,t}(x)=K_{0,t}(0)$ for $t\ge \tau_0(x)$.
\end{proof}
\vspace{0.2cm}



\section{Wiener chaos expansion}

When the canonical flow is filtered with respect to  $\cF^W$, it can be be expanded into a series of iterated Wiener integrals, see e.g. \cite[pg.57]{MR2060298} and \cite{MR1905858}. In this section, we derive the Wiener chaos expansion of $\mathbb{E}[f(X_t)|\cF^W]$ using the semigroup $P$ of the sticky Brownian motion.   This semigroup can be obtained explicitly by the inverse Laplace transform of the resolvent \cite[Prop.13]{Warren1997} for $x,y \in \mathbb{R}_+ $ and $t> 0$ as
\begin{equation} \label{St}
P_t(x,dy) = p_t(x,y)\, dy-p_t(x,-y)\, dy+2g_t(x+y)\, dy+\frac{1}{\theta}\: g_t(x) \,\delta_0(dy)
\end{equation}
where $p_t(x,\cdot)$ is the probability density function of a Gaussian random variable with mean $x$ and variance $t$, and
\[
g_t(x)=\theta \exp(2\theta x + 2\theta^2 t) \:\mbox{erfc} \left(\frac{x}{\sqrt{2t}}+\theta \sqrt{2t}\right)
\]
with $\mbox{erfc} (x) = \frac{2}{\sqrt{\pi}} \int_x^\infty e^{-y^2} dy$ \cite[Cor.3.11]{kost}.
The following lemma will be useful for deriving an equation to iterate for Wiener chaos expansion. Let $\cS:=\{f:[0,\infty)\to  \RR : f\in C^2(0,\infty), f(0^+), f'(0^+)\mbox{ and } f''(0^+) \mbox{ are finite},  \lim_{x\to \infty}f(x)=0\}$.
\begin{lemma} If $f\in \cS$, then $P_tf\in \cD(\cA)$  and $(P_t f)'1_{(0,\infty)} \in \cS$ for each $t> 0$.  \label{lemWChaos}
\end{lemma}
\begin{proof} Note that $p_t(x,y)=p_t(0,y-x)$ and  $p_t(x,-y)=p_t(0,x+y)$ and let $p_t'$ denote the derivative of $p_t(0,\cdot)$. For $x\geq 0$, we have
\begin{eqnarray*}
\lefteqn{(P_t f)'(x)=-\int_0^\infty p_t'(0,y-x)f(y)\, dy -\int_0^\infty p_t'(0,x+y)f(y)\, dy  }\\ & \displaystyle{\quad \quad \quad \quad +2\int_0^\infty g_t'(x+y)f(y)\, dy + \frac{1}{\theta}\, g_t'(x)f(0)}
\end{eqnarray*}
and
\begin{eqnarray*}
\lefteqn{(P_t f)''(x)=\int_0^\infty p_t''(0,y-x)f(y)\, dy -\int_0^\infty p_t''(0,x+y)f(y)\, dy  }\\ & \displaystyle{\quad \quad \quad \quad +2\int_0^\infty g_t''(x+y)f(y)\, dy + \frac{1}{\theta}\, g_t''(x)f(0)}\; .
\end{eqnarray*}
By a change of variable $y-x$ to $y$  in the first integral in the expression for $(P_t f)'$ above, we get
\begin{eqnarray*}
\lefteqn{(P_t f)'(x)=-\int_{-x}^\infty p_t'(0,y)f(y+x)\, dy -\int_x^\infty p_t'(0,x+y)f(y)\, dy  }\\ & \displaystyle{\quad \quad \quad \quad + 2\int_x^\infty g_t'(y)f(y-x)\, dy + \frac{1}{\theta}\, g_t'(x)f(0)}\; .
\end{eqnarray*}
Then, $\lim_{x\rightarrow \infty } (P_t f)'(x)=0$ since  $f$, $p'_t$, and $g'_t $ all vanish at infinity, $f$ is bounded, and the third integral above also goes to 0 as $x \to \infty$.  Moreover, $(P_t f)'(0)$,
 %
$(P_t f)''  (0)$, $(P_t f)'''  (0)$ are all finite since $p_t$, $g_t$ and their derivatives are continuous and bounded. It follows that  the function $(P_t f)'1_{(0,\infty)} $, and its first and second derivatives are all finite at $0^+$. Hence, $(P_t f)'1_{(0,\infty)}  \in \cS$.

On the other hand, one can easily verify that $$g''_t(x)=2\theta\, g'_t(x)+2\theta\, (x/\sqrt{2\pi  t^3})\exp(-x^2/ {2t})\; .$$ In view of this and the identity $p'_t(0,y)=(-y/t)p_t(0,y)$,  we get $(P_t f)''(0+)=2 \theta \, (P_t f)'(0+)$. The other properties in $\cD(\cA)$ are also satisfied by $P_t f$ and  $P_t f \in \cD(\cA)$ follows.
\end{proof}

\begin{proposition}  \label{prop2} For $f \in \cS$, we have
\[
\mathbb{E}[f(X_t)|\cF^W]=P_tf(0)+\sum_{n=1}^{\infty} J^n_t f
\]
where
$$ J^n_t f=\int_{0<s_1<\ldots<s_n<t}P_{s_1}(D(P_{s_2-s_1}\ldots D(P_{t-s_n}f)))(0)\, dW_{s_1}\ldots dW_{s_n}$$
and  $Dg=1_{(0,\infty)} g'$.
\end{proposition}

\begin{proof}
Let $H(s,x):=P_{t-s}f(x)$ for  $0<s<t$, $x\geq 0$ and $f\in \cS$, and recall that
\begin{eqnarray*}
\lefteqn{P_{t-s} f (x)=\int_0^\infty p_{t-s}(0,y-x)f(y)\, dy -\int_0^\infty p_{t-s}(0,x+y)f(y)\, dy  }\\ & \displaystyle{\quad \quad \quad \quad +2\int_0^\infty g_{t-s}(x+y)f(y)\, dy + \frac{1}{\theta}\, g_{t-s}(x)f(0)} \; .
\end{eqnarray*}
By continuity of $p$ and $g$, it follows that $H$ is differentiable when $f$ is measurable and bounded, in particular when $f\in \cS$. By It\^{o}'s formula for $H(s,X_s)$, we get
\begin{eqnarray*}
H(s,X_s)&=& H(0,0)+\int_0^s \frac{\partial}{\partial u}H(u,X_u)\, du   \\
& &  +\int_0^s \frac{\partial}{\partial x}H(u,X_u)\, dX_u+\frac{1}{2}\int_0^s \frac{\partial^2}{\partial x^2} H(u,X_u)\, d\langle X_u\rangle  \\
 &=& H(0,0)+\int_0^s \left[ (P_{t-u}f)'(X_u) \right] 1_{\{X_u>0\}} \, d W_u  \\ & & +\int_0^s \cA(P_{t-u}f)(X_u)\, du  + \int_0^s \left[\frac{d}{d u}P_{t-u}f\right](X_u)\, du
\end{eqnarray*}
where we have identified $\cA$  as $P_{t-u} f \in \cD(\cA)$ by Lemma \ref{lemWChaos}.
Then, the sum of the last two terms is 0 because for a transition semigroup $P$
\[
\frac{d}{dt} P_t f =  \lim_{v\to 0} \frac{1}{v}[P_{t+v}f -P_tf]  = \lim_{v\to 0} \frac{1}{v}[P_v(P_tf)-P_tf]     = \cA(P_t f)
\]
 by definition of the generator of an infinitesimal semigroup;
in particular,  $\frac{d}{du} P_{t-u}f =-\cA(P_{t-u} f)$. So, we have
\[
H(s,X_s) = H(0,0)+\int_0^s \left[ (P_{t-u}f)'(X_u) \right] 1_{\{X_u>0\}} \, d W_u   \; .
\]
By letting $s\uparrow t$, it follows that
\[
f(X_t)=P_t f(0)+ \int_0^t (P_{t-u}f)'(X_u) 1_{\{X_u>0\}} \, d W_u \; .
\]
 By conditioning with respect to $\cF^W$ and interchanging conditional expectation and integration (see e.g. \cite[Lem.4.7]{MR2235172}), we get
\begin{equation}\label{tr}
\mathbb{E}[f(X_t)|\cF^W]=P_t f(0) + \int_{0}^{t} \mathbb{E}[(P_{t-u} f)'1_{(0,\infty)} (X_u)|\cF^W] \, dW_u \; .
\end{equation}
Note that since the integrand is adapted and  the quadratic variation of $W$ is an absolutely continuous function of $t$, trivially,  the stochastic integral can be defined uniquely in almost sure sense (see e.g. \cite[Rem.3.2.11]{karat}). That is, we can work with a  progressively measurable modification of $\mathbb{E}[(P_{t-u} f)' 1_{(0,\infty)}(X_u)|\cF^W]$.

Let the Wiener chaos expansion of $\mathbb{E}[f(X_t)|\cF^W]$ be given by
\[
\mathbb{E}[f(X_t)|\cF^W]=P_t f(0)+\sum_{n=1}^{\infty}J_t^n f
\]
which exists in $L^2$ sense \cite[pg.202]{revuz}.
Now, in view of Lemma \ref{lemWChaos} $(P_t f)'1_{(0,\infty)} \in \cS$ and we can iterate Equation (\ref{tr}) to get
\[
J^1_t f = \int_0^t \int_{(0 ,\infty)} P_s(0,dx) (P_{t-s}f)'(x) \: dW_s
\]
and similarly
$$ J^n_t f=\int_{0<s_1<\ldots<s_n<t}P_{s_1}(D(S_{s_2-s_1}\ldots D(P_{t-s_n}f)))(0)\, dW_{s_1}\ldots dW_{s_n}$$
where  $Dg=1_{(0,\infty)} g'$.
\end{proof}



Note that Proposition \ref{prop2} uniquely characterizes the conditional law of $X_t$ given $W$ since $\cS$ is dense in $C_0([0,\infty))$. It also gives the Wiener chaos expansion for $G_f (W^+_ t )$ as well by Theorem \ref{ty}.  We can alternatively consider the semigroup for $W^+$, denoted by $P^+$ to obtain an expansion for $G_f(W_t^+)$. By similar calculations as above,  we find
 \begin{equation}  \label{WienerP+}
 G_f(W_t^+)=P^+_tG_f(0) + \sum_{n=1}^\infty J^{n+}_t f
\end{equation}
where
\[
J^{n+}_t f =\int_{0<s_1<\ldots<s_n<t}P^+_{s_1}(D^+(P^+_{s_2-s_1}\ldots D^+(P^+_{t-s_n}f)))(0) \, dW_{s_1}\ldots dW_{s_n}
\]
and $D^+g(x)= g'(x)$. Then,  the Wiener chaos expansions of $\mathbb{E}[f(X_t)|\cF^W]$ and $G_f(W^+_t)$ must be  equal. In particular,
$P_t f(0)=P^+_tG_f(0)$, and $J^{n}=J^{n+}$ for $n\ge 1$.
Vice versa, showing directly the equality of the Wiener chaos expansions of $\mathbb{E}[f(X_t)|W]$ and $G_f(W^+_t)$ would be an alternative approach to verify  the conditional law.  Instead, our proof in this paper has drawn upon the broader perspective of the generalized equation satisfied by flows induced by the sticky equation and the Wiener flow is completely described.


\end{document}